\documentclass[12pt]{amsart}

\usepackage{breakurl}
\usepackage[hyphens]{url}

\usepackage{upgreek}
\usepackage{amssymb}
\usepackage{mathtools,arydshln}
\usepackage[utf8]{inputenc}

\usepackage{mathrsfs}
\usepackage[shortlabels]{enumitem}

\usepackage[dvipsnames]{xcolor}

\usepackage[pagebackref,colorlinks,linkcolor=Maroon,citecolor=MidnightBlue,urlcolor=NavyBlue,hypertexnames=true]{hyperref}
\usepackage{showtags}

\usepackage{xspace,titletoc,microtype}

\usepackage{breqn}

\usepackage{bbm}

\allowdisplaybreaks 

\newcommand{\df}[1]{{\bf{#1}}{\index{#1}}}

 \numberwithin{equation}{section}

      \newtheorem{theorem}{Theorem}[section]

      \newtheorem{remark}[theorem]{Remark}
       \newtheorem{corollary}[theorem]{Corollary}
      \newtheorem{lemma}[theorem]{Lemma}
      
      \newtheorem{proposition}[theorem]{Proposition}
      
      \def\C{{\mathbb R}}
      \def\N{{\mathbb N}}
      \def\C{{\mathbb C}}

      \def\cU{\mathcal U}

      \def\cM{\mathcal M}

      \def\cW{\mathcal W}

      \def\cD{\mathcal D}

\newcommand{\rR}{{\tt{r}}}
\newcommand{\tT}{{\tt{t}}}
\newcommand{\mfZ}{\mathfrak{Z}}
\newcommand{\jj}{{\tt{a}}}
\newcommand{\kk}{{\tt{b}}}

\newcommand{\RR}{\mathbb{R}}
\newcommand{\CC}{\mathbb{C}}
\newcommand{\bS}{\mathbb{S}}
\newcommand{\cZ}{\mathcal{Z}}
\newcommand{\cL}{\mathcal{L}}

\newcommand{\vg}{{\tt{g}}}
\newcommand{\tg}{{\tt{h}}}

\newcommand{\vmu}{\mu}
\newcommand{\vnu}{\mu}
\newcommand{\ka}{\kappa}

\newcommand{\SSnmu}{\mathbb S_n(\C^\vmu)}
\newcommand{\SSn}{\mathbb{S}_n(\C)}
\newcommand{\SSnvg}{\mathbb{S}_n(\C^\vg)}
\newcommand{\SSnk}{\mathbb{S}_n(\C^k)}

\newcommand{\SSNmu}{\mathbb{S}_N(\C^\mu)}
\newcommand{\wtc}{\widetilde{c}}
\newcommand{\spann}{\operatorname{span}}
\newcommand{\AB}{C}
\newcommand{\vv}{\gamma}

\makeindex

\title[BMIs and free polynomials]{Bilinear matrix inequalities and polynomials in several freely noncommuting variables}

\author[S. Balasubramanian]{Sriram Balasubramanian${}^*$}
\address{Department of Mathematics\\
IIT Madras, Chennai - 600036, India.}
\email{bsriram@iitm.ac.in, bsriram80@yahoo.co.in}
\thanks{${}^*$ Supported by the grant MTR/2018/000113 from the Department of Science and Technology (DST),
 Govt. of India.}

\author[N. Hotwani]{Neha Hotwani${}^1$}
\address{Department of Mathematics\\
IIT Madras, Chennai - 600036, India.}
\email{ma18d016@smail.iitm.ac.in}
\thanks{${}^1$ Supported by the fellowship 0203/16(8)/2018-R\&D-II from the National Board for Higher Mathematics (NBHM),
 Govt. of India.}

\author[S. McCullough]{Scott McCullough}
\address{Scott McCullough, Department of Mathematics\\
 University of Florida\\ Gainesville 
  }
  \email{sam@math.ufl.edu}

\subjclass[2010]{46N10, 26B25 (Primary); 47A63, 52A41, 90C25 (Secondary)}

\keywords{partial convexity, biconvexity, bilinear matrix inequality (BMI), 
noncommutative polynomial}

\begin{document}

\begin{abstract}
 Matrix-valued polynomials in any finite number of
 freely noncommuting variables that
 enjoy certain canonical partial convexity properties are characterized,
 via an algebraic certificate, 
 in terms of Linear Matrix Inequalities and Bilinear Matrix 
 Inequalities.   
\end{abstract}

\maketitle

\section{introduction}
 The main results of this article extend  principal results of
 \cite{HHLM08} on convex
  polynomials in freely noncommuting variables to the matrix-valued case 
  and of \cite{JKMMP21} on $xy$-convex polynomials  
 to the  matrix-valued setting  in any finite number of freely
 noncommuting variables.

 Fix a positive integer $\vg.$ Given a positive integer $d$ 
 and 
 $d\times d$ matrices, $A_0,A_1,\dots,A_\vg,$ the expression 
\[
 L_A(x) = A_0 - \sum_{j=1}^{\vg} A_j x_j 
\]
 is a \df{linear pencil}, where $A = (A_0,A_1,\dots,A_g).$ \index{$L_A$}
 In the case the $A_j$ are hermitian the pencil is \df{hermitian}
 and,  in this case, it is typically assumed that $A_0$
 is positive definite. 
 When $L_A$ is hermitian and  $x\in \RR^g,$ the matrix $L_A(x)$ is hermitian
 and
\[
 L_A(x)\succeq 0
\]
 is a \df{linear matrix inequality (LMI)}. Here  $T\succeq 0$
 indicates that the hermitian matrix $T$ is positive semidefinite.
   The (scalar) solution, or \df{feasible}, set of a 
 hermitian pencil $L_A,$ 
\[
 \cD_A[1] =\{x\in \RR^{\vg}: L_A(x) \succeq 0\},
\]
 is a \df{spectrahedron}. 
 Because $L_A$ is affine linear, it is evident that \df{$\cD_A[1]$} is convex.
 Spectrahedra figure prominently in numerous engineering
 applications. They are fundamental objects in semidefinite programming
  in convex optimization and in 
 real algebraic geometry. 

 Given $d\times d$ hermitian matrices 
 $A_0,A_1,\dots,A_{\vg},B_1,  \dots, B_{\tg}, C_{pq}, 1 \le p \le \vg, 1 \le q \le \tg,$
 the expression 
\begin{equation*}
L(x,y) = A_0 - \sum_{j=1}^{\vg} A_j x_j - \sum_{k=1}^{\tg} B_k y_k 
 -  \sum_{p,q=1}^{\vg,\tg} C_{pq} x_p y_q,
\end{equation*}
 is an \df{$xy$-pencil}. 
 When 
 all the coefficient matrices are hermitian, $L$
 is a \df{hermitian $xy$-pencil}. For a hermitian 
 $xy$-pencil, the inequality 
 $L(x,y)\succeq 0$ is a \df{Bilinear Matrix Inequality (BMI)}. 
 Bilinear matrix inequalities appear in robust control.
 See for instance \cite{KSVdS,SGL,vAB} and the references
 therein and the MATLAB toolbox, \\
{\url{https://set.kuleuven.be/optec/
 Software/bmisolver-a-matlab-package-for-solving-optimization-problems
 -with-bmi-constraints.}}

 It is natural from multiple perspectives to consider 
 the fully matricial analogs of LMIs and BMIs. 
 For positive integers $n,$ let \df{$\SSn$} denote 
 the set of $n\times n$ hermitian matrices and let
 \df{$\SSnvg$} denote the set of $\vg$-tuples from $\SSn.$
 Given $X=(X_1,\dots,X_\vg)\in \SSnvg,$ let
\[
  L_A(X) =A_0\otimes I_n -\sum A_j\otimes X_j
\] 
 and let 
\[
 \cD_A[n]=\{X\in \SSnvg: L_A(X)\succeq 0\}.
\]
 The sequence $\cD_A=(\cD_A[n])_n$ is known as a
  \index{$\cD_A$}  \df{free spectrahedron}
 or \df{LMI domain}.  While $\cD_A[1]$ does not determine $A,$
 up to unitary equivalence, the free spectrahedra $\cD_A$ does.

 Free spectrahedra are \df{matrix convex}, meaning 
\begin{enumerate}[(i)]
 \item $\cD_A$ is closed with respect to isometric compressions:
 if $X\in \cD_A[n]$
 and $V:\CC^m\to \CC^n$ is an isometry, then $V^*XV\in \cD_A[m],$
 where 
\[
 V^*XV = V^*(X_1,\dots,X_\vg)V = (V^*X_1V, \dots, V^*X_\vg V);
\]
 and 
\item $\cD_A$ is closed under direct sums:
 if  $X\in \cD_A[n]$ and $Y\in \cD_A[m],$ then $X\oplus Y\in \cD_A[n+m],$
 where
\[
 (X\oplus Y)_j =\begin{pmatrix} X_j&0\\0& Y_j \end{pmatrix}.
\]
\end{enumerate}
 In particular each $\cD_A[n]$ is convex in the ordinary sense. 

 Free spectrahedra appear in the theories of
 completely positive maps and operator systems and spaces 
 \cite{paulsen,pisier}. They appear in systems engineering problems
 governed by a signal flow diagram as explained in 
 \cite{convert-to-matin,IMA-survey,CHSY}.
 They also  produce tractable natural relaxations for optimizing 
 over  spectrahedra; e.g., 
 the matrix cube problem  \cite{BtN,DDSS,matcube},  which can be NP hard, but 
 whose canonical free spectrahedral relaxation is a semidefinite program (SDP).

 The fully matricial analog of BMIs is described below,
 after the introduction of polynomials in freely noncommuting
 variables. 
 
\subsection{Free polynomials}
 The two types of partial convexity considered in this article are described
 in terms of free polynomials.  Fix freely noncommuting variables
 $\chi_1,\dots,\chi_k.$ Given a word 
\begin{equation}
 \label{d:w}
  w=\chi_{i_1} \cdots \chi_{i_\ell}
\end{equation}
 in these variables and $T\in \SSnk,$ let
\[
 w(T)=T^w = T_{i_1}\cdots T_{i_\ell}.
\]
 Let $\cW$ denote the collection of words in the variables $\chi.$
 A  \df{$d \times d$ matrix-valued free polynomial} is an expression of the form 
\[
 p(\chi)= \sum_{w\in \cW} p_w w,
\]
 where the sum is finite and the $p_w \in M_d(\C)$. 
 The free polynomial $p$ is naturally \df{evaluated} at $T\in \SSnk$ as
\[
 p(T) =\sum p_w \otimes T^w.
\]
 There is a natural \df{involution ${}^*$} on free polynomials that 
 reverses the order of products in words so that, for $w$ 
 in equation~\eqref{d:w}, 
\[
 w^* = \chi_{i_\ell}\cdots \chi_{i_1}; 
\]
 and such that
\[
 p^* =\sum p_w^* w^*.
\]
 This involution is compatible with the adjoint operation on matrices,
\[
 p(T)^* =p^*(T).
\]
 A free polynomial $p$ is {\bf hermitian}
 \index{hermitian polynomial}
  if $p^*=p$; equivalently, if $p(T)^*=p(T)$ 
 for all $n$ and $T\in \SSnk.$

 From here on we often omit the adjectives matrix and  free and simply
 refer to matrix-valued free polynomials as polynomials, particularly when
 there is no possibility of confusion.

 Since the involution fixes the variables, $\chi_j^*=\chi_j,$
 we refer to $\chi_1,\dots,\chi_k$ as \df{hermitian variables}.
 In Subsection~\ref{s:BVMM}, non-hermitian variables 
 naturally appear.

\subsection{Partial convexity}
 Both types of partial convexity considered in this article involve 
 partitioning  freely noncommuting variables 
 into two classes  $x_1,\dots,x_{\vmu}$ and $y_1,\dots,y_{\vmu}.\footnote{ 
 For the results here, there is no loss in
 generality in assuming the number of $x$ and $y$ variables
 is the same.}$

\subsubsection{$xy$-convexity}
Since matrix multiplication does not commute,
 we now update the definition of an $xy$-pencil as follows. (See \cite{JKMMP21a}.)
 A matrix-valued free polynomial of the form
\begin{equation*}
 L(x,y) = A_0  - \sum_{j=1}^{\vmu} A_j x_j - \sum_{k=1}^{\vnu} B_k y_k 
 -  \sum_{p,q=1}^{\vmu} C_{pq} x_p y_q -   
   \sum_{p,q=1}^{\vmu} D_{qp} y_q x_p,
\end{equation*}
 where $A_j,B_k,C_{pq},D_{qp}$ are all matrices of the same size,
 is an \df{$xy$-pencil}.  The pencil $L$
  is naturally evaluated at a tuple
  $(X,Y)\in \SSnmu \times \SSnmu$ as
\[
 L(X,Y) = A_0 \otimes I_n - \sum_{j=1}^{\vmu} A_j \otimes X_j - \sum_{k=1}^{\vnu} B_k \otimes Y_k 
 -  \sum_{p,q=1}^{\vmu} C_{pq} \otimes X_p Y_q -    \sum_{p,q=1}^{\vmu} D_{qp} \otimes Y_q X_p.
\]
  When the $A_j$ and $B_k$ are hermitian and $D_{qp}=C_{pq}^*,$ the pencil
 $L$ is a \df{hermtian $xy$-pencil} and $L(X,Y)\succeq 0$ is the matricial analog of a BMI.
 Assuming, as we usually do, $A_0$ is positive definite, 
 writing $\Sigma=(A_j,B_j, C_{ij})$ and 
  $L_\Sigma=L,$ let
\[
 \cD_\Sigma[n] =\{(X,Y)\in \bS_n^{\vmu}\times \bS_m^{\vnu}: L_\Sigma(X,Y)\succeq 0\}
\]
 and let \df{$\cD_\Sigma$}
  denote the sequence $(\cD_\Sigma[n])_n.$ The set
 $\cD_\Sigma$ is \df{$xy$-convex}, meaning $\cD_\Sigma$ is
\begin{enumerate}[(i)]
 \item 
   closed  under direct sums; and
 \item 
    if $(X,Y)\in \cD[n]$ and
 $V:\CC^m\to \CC^n$ is an isometry such that $V^*(X_iY_j)V=V^*X_iVV^*Y_jV,$ 
 for all $i,j,$ 
 then $V^*(X,Y)V\in \cD_\Sigma[m].$ 
\end{enumerate}

 A tuple $((X,Y),V)$ where $(X,Y)\in \SSnmu\times \SSnmu$ and
 $V:\CC^m\to \CC^n$ is an isometry 
 such that $V^*(X_iY_j)V=V^*X_iVV^*Y_jV,$ 
 for all $i,j,$  is an \df{$xy$-pair}.   A hermitian matrix-valued
 free polynomial $p(x,y)$ is \df{$xy$-convex} if  
\[
p(V^*(X,Y)V)\preceq (I_d \otimes V)^*p(X,Y) (I_d \otimes V)
\]
 for all  xy-pairs $((X,Y),V).$ It is nearly immediate that,
 if $p$ is $xy$-convex, then the positivity set of $-p,$
\[
 \cD_{-p}=\{(X,Y): p(X,Y)\preceq 0\},
\]
 is also $xy$-convex. \index{$\cD_{-p}$} 
  Theorem ~\ref{t:introxyconvexp} below provides
 an algebraic certificate 
 characterizing $xy$-convex polynomials. 
 When $d=\vmu=1,$ it reduces to \cite[Theorem~1.4]{JKMMP21}.

\begin{theorem}
\label{t:introxyconvexp}
Suppose $p(x,y)$ is a hermitian $d \times d$ matrix-valued  polynomial. If $p$ is  $xy$-convex,
 then there exist a hermitian $d \times d$ matrix-valued $xy$-pencil $\lambda,$
 a positive integer $N$  and an  $N \times d$ matrix-valued $xy$-pencil 
 $\Lambda$ such that 
\begin{equation}
\label{eq:xy-conv}
 p(x,y) = \lambda(x,y) +\Lambda(x,y)^* \Lambda(x,y).
\end{equation}
 In particular, $-p$ is the Schur complement 
 of a Hermitian $xy$-pencil and 
 $\cD_{-p}$  is the feasible  set of  the BMI,
\[
  \begin{pmatrix} I & \Lambda(x,y)\\\Lambda(x,y)^* 
    &  -\lambda(x,y) \end{pmatrix} \succeq 0.
\]

The converse is easily seen to be true.
\end{theorem}

A proof of Theorem \ref{t:introxyconvexp} is contained in the proof of Proposition 
\ref{p:more-is-true} given in Section 3.

\subsubsection{$a^2$-convexity}
 To maintain consistency with the literature, we now switch
 to freely noncommuting variables $a_1,\dots,a_\vmu$ and 
 $x_1\, \dots,x_{\vnu}.$
 A $d\times d$ matrix-valued  hermitian  polynomial  $p(a,x)$ is \df{convex in $x$}
  if for each positive integer $n,$
  each   $A \in \mathbb S_n(\C^\mu),$ each $X,Y \in \mathbb S_n(\C^\mu)$ and 
  each   $0 < t < 1,$ one has 
\[
  p(A, tX + (1-t)Y) \preceq t p(A,X) + (1-t)p(A,Y).
\]
  A canonical  example of a  convex in $x$ 
  polynomial is a hermitian \df{linear in $x$ pencil}; 
 that is, a hermitian polynomial that is affine linear in $x.$

 There is a fruitful alternate characterization of convexity in $x.$ 
 A tuple $((A,X),V)$ where  $(A,X)\in \SSnmu\times\SSnmu$
 and $V:\CC^m\to\CC^n$  is an  isometry 
 is an \df{$a^2$-pair} if $V^*A_i^2V= (V^*A_iV)^2$ 
 for each $1\le i\le \vmu.$ Equivalently 
 $((A,X),V)$ is an $a^2$-pair if
 $\operatorname{ran} V$ reduces $A.$ 
  As we will see in  Proposition~\ref{thm:partialconvexity}, 
 a hermitian
  polynomial $p$ is convex in $x,$ or \df{$a^2$-convex},  if  and only if 
\[
 p(V^*(A,X)V) \preceq (I_d\otimes V^*)p(A,X) (I_d\otimes V)
\]
 for all $a^2$-pairs $((A,X),V).$
  Theorem~\ref{t:introxyconvexp} and Theorem~\ref{thm:partialconvexpolys} below --
 the latter of
  which is a matrix polynomial version of \cite[Theorem~1.5]{HHLM08} 
 and \cite[Corollary~1.3]{JKMMP21} -- 
 are the main results of this article.

\begin{theorem}
\label{thm:partialconvexpolys}
 Suppose $p(a,x)$ is a $d\times d$ matrix-valued hermitian polynomial.  If 
$p(a,x)$ is convex in $x,$  then there exist 
 a $d\times d$ matrix-valued hermitian linear in $x$ pencil $L,$
 a positive integer $N$ and a $N\times d$  matrix-valued 
 polynomial $\Lambda$ that is linear in $x$ such that
\[
  p(a,x) = L(a,x) + \Lambda(a,x)^* \Lambda(a,x).
\]
 In particular, $p$ has degree at most two in $x$ and
 $\cD_{-p}$ is the feasible
 set of the affine linear in $x$ matrix inequality
\[
 \begin{pmatrix} I & \Lambda(a,x)\\ \Lambda(a,x)^* & - L(a,x)
 \end{pmatrix}  \succeq 0.
\]
 
 The converse is evidently true.
\end{theorem}

A  proof of Theorem \ref{thm:partialconvexpolys} is given in 
 Section~\ref{s:a2-convexity}. 
Proposition~\ref{p:more-is-true} below
 describes the relationship between 
$xy$-convexity and separate convexity in $x$ and $y$.
It also  extends
 \cite[Theorem~1.4]{JKMMP21} to both several $x$ and $y$
 variables and matrix-valued polynomials.

\begin{proposition}
 \label{p:more-is-true} 
Let $p(x,y)$ be a $d \times d$ matrix-valued hermitian polynomial. The following 
statements are equivalent.
\begin{enumerate}[(i)]
\item \label{i:more1} $p$ is $xy$-convex
\item \label{i:more2} $p$ is convex in $x$ and $y$ separately. 
\item \label{i:more3} $p$ has the form given in equation \eqref{eq:xy-conv}.
\end{enumerate}
In particular, $p$ is $xy$-convex if and only if $p$ is  
 convex in $x$ and $y$ separately.
\end{proposition}

A proof of Proposition \ref{p:more-is-true} is given in Section 3.

\begin{remark}\rm
 An example in the appendix of the arxiv version of \cite{JKMMP21}
 shows that there is not a local version of Proposition~\ref{p:more-is-true}.
 That is, as a local statement, separate convexity need not imply
 $xy$-convexity.
\end{remark}

\section{Partially convex hermitian matrix-valued NC polynomials}
\label{s:a2-convexity}
 This section contains a proof of Theorem~\ref{thm:partialconvexpolys}
 and is organized as follows.
 Subsection~\ref{s:a2convex} presents alternate formulations of $a^2$-convexity.
 Needed versions of  Amitsur's no polynomial identities results 
 are collected in  Subsection~\ref{s:amitsur}. 
 The border vector middle matrix representation for a type of Hessian
 for polynomials in $a,x$ of degree two in $x$ is reviewed in
 Subsection~\ref{s:BVMM}.
  The proof of Theorem~\ref{thm:partialconvexpolys}
 concludes in Subsection~\ref{s:a2endproof}. Subsection~\ref{s:a2cors} 
 contains two corollaries 
 that apply to $xy$-convex polynomials.

\subsection{Alternate formulations of convexity}
\label{s:a2convex}
The proof of Theorem~\ref{thm:partialconvexpolys} makes use of the  
following characterization of $a^2$-convex polynomials.
It parallels \cite[Proposition~4.1]{JKMMP21} for
$xy$-convex polynomials for $\vmu=1$ and, to some extent, appears as 
\cite[Proposition~1.5]{JKMMP21}. It also borrows liberally
 from the ideas in \cite{royal}.

\begin{proposition}
\label{thm:partialconvexity}
 For a $d \times d$ matrix-valued hermitian polynomial $p(a,x),$ 
  the following statements are equivalent. 
\begin{enumerate}[(i)]
\item  \label{i:a2convex1}
 The polynomial $p$ is convex in $x;$
\item \label{i:a2convex2}
  If $((A,X),V)$ is an $a^2$-pair, then 
\[
(I_d \otimes V)^*\, p(A,X) \, (I_d \otimes V)  \succeq p(V^*(A, X)V);
\]
\item \label{i:a2convex3}
  For each tuple $(A,X) \in \SSnmu\times\SSnmu,$
each positive integer $m$ and 
 all tuples $\alpha,\delta\in \mathbb S_m(\C^\mu)$
and $\beta \in M_{n,m}(\C^\mu),$
\[
 (I_d \otimes W)^* p(R,S) (I_d \otimes W) 
\succeq  p \left(W^* (R,S) W \right)
\] 
where $W^* = \begin{pmatrix} I_n & 0 \end{pmatrix} \in M_{n, n+m}(\C),$ 
\[
R = \left(\begin{pmatrix} A_1 & 0 \\ 0 &  \alpha_1 \end{pmatrix},\dots, \begin{pmatrix} 
A_\mu & 0 \\ 0 &  \alpha_\mu \end{pmatrix}\right) 
 \in \mathbb S_{n+m}(\C^\mu) 
\]
and
\[
S = \left(\begin{pmatrix} X_1 & \beta_1 \\ \beta_1^* &  \delta_1 \end{pmatrix},\dots,
\begin{pmatrix} X_\mu & \beta_\mu \\ \beta_\mu^* &  \delta_\mu \end{pmatrix}\right) 
 \in  \mathbb S_{n+m}(\C^\mu).
\]
\end{enumerate}
\end{proposition}

\begin{proof}
 To prove \ref{i:a2convex2} implies \ref{i:a2convex1},
   let  $A, X,Y \in \SSnmu$
  and $t \in [0,1]$ be given.  Let 
\[
 \widehat A = \begin{pmatrix} A & 0 \\ 0 & A
\end{pmatrix}, \ \ \
  \widehat X = \begin{pmatrix}  X & 0 \\0 & Y  \end{pmatrix},  
  \ \ \ V=  \begin{pmatrix}  
      \sqrt{t} \, I_n &  \sqrt{1 - t} \, I_n \end{pmatrix}^*.
\] 
  In particular, $((\widehat A,\widehat X),V)$ is
 an $a^2$-pair. Thus,
\begin{equation*}
\begin{split}
 p(A, tX + (1-t) Y) 
 & = p(V^*( \widehat A, \widehat X)V) \\
 & \preceq (I_d \otimes V)^*\, p(\widehat A,\widehat X)\, (I_d \otimes V)
\\  & = (I_d \otimes V)^* \begin{pmatrix} p(A, X) & 0 \\ 0 & p(A, Y) \end{pmatrix} (I_d \otimes V) \\
& = t p(A,X) + (1-t) p(A, Y),
\end{split}
\end{equation*}
 where the inequality is a consequence of the hypothesis.
 Hence $p$ is convex in $x.$

Now suppose item~\ref{i:a2convex3} holds and let an $a^2$-pair
 $((A,X),V)$ be given. 
 Since $V: \C^m \rightarrow \C^n$ is an isometry whose range $M$ reduces $A_j,$ 
 the matrix representations of $V,$ $A_j$ and $X_j$ with respect to the decomposition
  $\C^n = M \oplus M^{\perp}$ take the forms
\[
 \begin{pmatrix}  I_M \\ 0  \end{pmatrix}, \begin{pmatrix} 
  A_j\restriction_M & 0 \\ 0 & 
   A_j\restriction_{M^\perp} \end{pmatrix} 
\text{ and } \begin{pmatrix}  P_{M}  X_j P_M^* & Y_j \\ Y_j^* & P_{M^{\perp}} X_j P_{M^\perp}^* \end{pmatrix} 
\]
respectively, where $P_M$ denotes the orthogonal projection of $\C^n$ onto $M.$ 
The conclusion of item~\ref{i:a2convex2} now follows by identifying $M$ with $\C^m$ 
and observing  that, under this identification, the operators $V,$ $A$ 
 and $X$ have the same form  as $W,$ $R$ and $S$ in the hypothesis.
 Hence item~\ref{i:a2convex3} implies item~\ref{i:a2convex2}. 

 It remains to prove \ref{i:a2convex1} implies \ref{i:a2convex3}. 
 To this end, let
\[
\widehat {S}  
= \left(\begin{pmatrix} X_1 &- \beta_1 \\ -\beta_1^* &  \delta_1 \end{pmatrix},\dots,
\begin{pmatrix} X_\mu & -\beta_\mu \\ -\beta_\mu^* &  \delta_\mu \end{pmatrix}\right).
\]
By the convex in $x$ hypothesis, it follows that 
\begin{equation}
 \label{e:a2convex1}
   \begin{pmatrix}  p(A, X) & 0 \\ 0 & p(\alpha, \delta) \end{pmatrix} = 
   p \left(R, \frac{1}{2} (S + \widehat{S})\right) \preceq
     \frac1 2 (p(R,S) + p(R, \widehat{S})).
\end{equation}
Multiplying the inequality  of equation~\eqref{e:a2convex1}
  by $(I_d \otimes W)^*$ on the left and $(I_d \otimes W)$ on the right gives
\[
p(A,X) = p(W^*(R,S)W) \preceq \frac{1}{2} (I_d \otimes W)^*
 \, [p(R,S) + p(R, \widehat{S})] \, (I_d \otimes W).
\]
Thus, to complete the proof, it suffices to show.
\begin{equation}
\label{e:complete}
(I_d \otimes W)^*\, p(R,\widehat{S}) \,  (I_d \otimes W) 
  = (I_d \otimes W)^* \, p(R,S) \, (I_d \otimes W).
\end{equation}
To this end, let 
\[
U=\begin{pmatrix} I & 0 \\0 & -I \end{pmatrix}
\]
and note  $(R,\widehat{S}) = U^* (R,S)U.$ Consequently, 
\[
 p(R,\widehat{S})  = (I_d \otimes U^*) p(R,S) (I_d \otimes U), 
\]
and  equation~\eqref{e:complete} follows.
\end{proof}

\subsection{Faithful representations}
\label{s:amitsur}
\begin{proposition}
 \label{p:a2convex-deg2}
  Suppose $p(a,x)$ is a
 hermitian polynomial. If $p$ is convex in $x,$ then
 the degree of $p$ in $x$ is at most two. 
\end{proposition}

\begin{proof}
 Let $d$ denote the size of $p.$ Thus 
  $p = \sum_w p_w w$ for some $p_w\in M_d(\C).$
 For $\gamma  \in \C^d,$ 
define the polynomial $p_{\gamma}$ by 
  $p_\gamma =  \sum_w (\gamma^* p_w \gamma ) \, w.$ 
Since $p$ is hermitian, it follows that $p_\gamma$ is a 
hermitian polynomial with scalar coefficients. Also  convexity of $p$ in $x$ implies the convexity of 
$p_{\gamma}$ in $x.$
Hence, by \cite[Corollary~1.3]{JKMMP21},\footnote{The same
 result, but with real scalars, appears as \cite[Theorem~1.4]{HHLM08}}
 for each $\gamma\in \C^d,$ the degree of $p_\gamma$ in $x$ is at most two. 
 Suppose the word $w = w(a,x)$   is such that $p_w \neq 0.$ 
 Since the scalar field is $\C,$ it follows that
 there exists a $\gamma\in \C^d$ such that $\gamma^*p_w\gamma\ne 0.$
 Since $p_\gamma$ has degree at most  two in $x,$ it follows that
 $w(a,x)$ has degree at most two in $x.$ Hence $p$ has
  degree at most two in $x.$
\end{proof}

The following lemma is a variant of the Amitsur-Levitski Theorem.

\begin{lemma}
\label{lem:vanishingpoly}
 If  $p(a)$ is a polynomial of degree at most $m\ge 0$ in the 
 freely noncommuting variables
 $a_1,\dots,a_\mu$ and if there is an 
 $n\ge N(\mu,m) := \sum_{j=0}^{m} \mu^j$ and a nonempty  
 open set $\cU\subseteq \SSnmu$ such that
 $p(U)=0$ for all $U\in \cU,$ then $p=0.$
\end{lemma}

%

\begin{proof} 
 Arguing by contradiction, suppose $p\ne 0,$
 but there is an $n\ge N(\mu,m)$ 
 and a nonempty open subset $\cU\subseteq \SSnmu$
 on which $p$ vanishes.  In this
 case, there is no loss of generality assuming 
 the degree of $p$ is $m.$ 
   Since $p$ vanishes on an open subset of $\SSnmu,$
 it vanishes on all of $\SSnmu.$ 

 For the moment, assume $n=N.$ Let $H$ denote the Hilbert
 space with orthonormal basis $\cW,$ the words of 
 length at most $m$ in the variables $a.$ Hence
 $\dim H=N.$  Define linear maps $S_j,$ for $j=1,\dots,\vg,$
 on $H$ by $S_j w= a_jw,$ if $w\in \cW$ has length 
 strictly less than $m,$ and $S_jw=0$ if $w$ has length $m.$
 Observe that  $S_j^*w$ has length strictly less than 
 the length of $w\in \cW.$

 Let $T_j=S_j + S_j^*.$ Thus 
 $T=(T_1,\dots,T_{\mu}) \in  \SSNmu.$ %
 A straightforward computation shows, when $m\ge 1$
\[
 p(T)\varnothing = \sum_{j=0}^{m-1}  q_j + p_m,
\]
 where $q_j$ are homogeneous polynomials of degree $j$
 and $p_m$ is the homogeneous of degree $m$ part of $p.$
 On the other hand, when $m=0,$
\[
 p(T)\varnothing = p_\varnothing \varnothing. 
\]
 Since the set $\{q_0, q_1,\dots,q_{m-1},p_m\}\subseteq H$ 
 is linearly independent and, by assumption, $p(T)\varnothing=0,$
 it follows that $p_m=0,$ contradicting the assumption
 that the degree of $p$ is $m.$ 

 To complete the proof, if $n>N,$ then replace 
 the tuple $T$ by $R=T\oplus 0$ and $\varnothing$
 with $\gamma =\varnothing \oplus 0,$ where
 the first $0$ is the zero tuple in $\mathbb{S}_{n-N}(\CC^\mu)$
 and the second $0$ is the zero vector in $\CC^{n-N},$ and
 observe that $0=p(R)\gamma$ implies $p(T)\varnothing =0.$
\end{proof}

\begin{lemma}
\label{cor:gsforus}
Let  $q(a) = \sum_w q_w  w(a)$ be a $d \times d$  matrix 
 (not necessarily hermitian) polynomial in the 
 freely noncommuting  variables $a_1\dots,a_\mu.$ 
 If $q(A) = 0$ for all $n \in \N$ and $A \in \SSnmu,$ then $q = 0.$
\end{lemma}

\begin{proof}
Since $q(a)$ is a $d \times d$  matrix polynomial,  
 i.e $q_w \in M_d(\C),$ it can be
 viewed as a $d \times d$ matrix $(q^{i,j}(a))_{i,j=1}^d$ of scalar 
  polynomials.
 Suppose that $q$ is nonzero. Choose $i,j$ 
 such that $q^{i,j}(a)$ is nonzero.
 Since $q^{i,j}(A) = 0$ for all $n \in \N$ and 
 $A \in  \mathbb S_n(\C^\mu),$ 
 Lemma \ref{lem:vanishingpoly} implies $q^{i,j}$ 
  is the zero
polynomial, a contradiction.
\end{proof}

\begin{proposition}
\label{rem:betterlinindep}
 For each positive integer $\ka$ and
 each $n\ge N=\sum_{j=0}^\ka \mu^j$  there exist $A \in \SSnmu$
  and $v \in \C^n$ such that 
\[
 \cM_{A,v,\ka}=\{w(A)v\,:\, 
    \text{$w(a)$ is a word  with degree at most $\ka$} \}
\]
is linearly independent.

 In particular, 
 in the case $\kappa=1,$ there is an $A\in \mathbb{S}_{\mu+1}(\C^\mu)$
 and a $v\in \C^{\mu+1}$ such that $\cM_{A,v,1}$ is linearly 
 independent.
\end{proposition}

\begin{proof}
 Fix $\ka.$
 Let $\cW_\ka$ denote the words in the (freely noncommuting)
 variables $a_1,\dots,a_\vmu$ of degree at most $\ka.$ 
 The cardinality of $\cW_\ka$ is $N=\sum_{j=0}^\ka \mu^j.$
 Given $c:\cW_\ka\to \C,$ let $c_w$ denote the value of 
 $c$ at $w\in \cW_\ka.$
  Let
 $\mathcal{C}$ denote the set of all functions $c:\cW_\ka\to \C$ such that
 $\sum_w |c_w|^2 =1.$ Thus $\mathcal{C}$ is identified with the unit
 sphere in $\CC^N$  and is thus compact.

  Given $c\in \mathcal{C},$ let
 \[
 q_c(a) = \sum_{w\in \cW_\ka} c_w w.
\]
  Arguing by contradiction suppose,  for each  $n \ge N$,
 for  each  $\AB\in \SSnmu$ and each $\vv\in \C^n,$ 
 there exists a $c\in \mathcal{C}$ such that 
 $q_c(\AB)\vv=0.$  Given  $n \ge N$ and $\AB\in \SSnmu$ and $\vv\in \C^n,$ let 
\[
 K_{\AB,\vv} =\{c\in \mathcal{C}: q_c(\AB)\vv=0\}.
\]
 Thus $K_{\AB,\vv}$ is nonempty for all $\AB$ and $\vv.$ 
 Likewise, since, for $\AB\in \SSnmu$ and $\vv\in \CC^n,$
  the mapping
\[
 \mathcal{C}\ni c \mapsto q_c(\AB)\vv\in \CC^n
\]
 is continuous, the sets $K_{\AB,\vv}$ are compact. 
 Given a positive integer $M$,  positive integers $n_1,\dots, n_M \ge N$,
 $\AB^j\in \mathbb{S}_{n_j}(\C^\vmu)$ and
 $\vv_j\in \C^{n_j}$ for $1\le j\le M,$ observe that
\[
 \cap_{j=1}^M K_{\AB^j,\vv_j}  = K_{\oplus \AB^j,\oplus \vv_j} \ne \emptyset.
\]
 Hence $\{K_{\AB,\vv}: \AB,\vv\}$ has the finite intersection property. 
 It follows that
\[
 \cap_{\AB,\vv} K_{\AB,\vv} \ne \emptyset.
\]
 Choosing any $\wtc$ in this intersection, 
\[
 q_\wtc(\AB)\vv=0
\]
for all $\AB$ and $\vv.$
 Consequently
  $q_\wtc(\AB)=0$ for all $\AB \in \SSnmu$ and hence, 
 by Lemma~\ref{lem:vanishingpoly}, $q_\wtc=0.$
 Thus, $\wtc_w=0$
 for all $\vv,$ contradicting $\wtc\in\mathcal{C}.$
 Hence, there exist $\AB$ and $\vv$ such that 
 $\cM_{\AB,\vv,\ka}$ is linearly independent.
 Let $\ell$ denote the size of $\AB;$ that is
 $\AB\in \mathbb{S}_\ell(\C^\mu)$ and $\cM_{\AB,\vv,\ka}$
 is a subspace of $\CC^\ell$ of dimension $N.$
  Let $V$ denote the inclusion
 of $\cM_{\AB,\vv,\ka}$ into $\CC^\ell$ and let
 $B=V^*\AB V.$ 
 Since $V^*A^\alpha V \vv= A^\alpha \vv\in \cM_{\AB,\vv,\ka}$
 for words $\alpha$ of length at most $\ka,$
 the set 
\[
 \{w(B)\vv: w \text{ is a word of length at most } \ka\}
\]
 is linearly independent.

 Given $m>N,$
 let $A=B\oplus 0,$ where $0\in M_{m-N}(\C^\mu).$
 Likewise let $v=\vv\oplus 0  \in\C^m = \CC^N \oplus \CC^{m-N}$
 and note that $\cM_{A,v,\ka}$ is linearly independent.
\end{proof}

\subsection{The Border vector,  middle matrix 
 and non-hermitian variables}
\label{s:BVMM}
 In this subsection, $q(a,x)$ denotes a fixed polynomial that
 is homogeneous of degree two in $x$ and $d_a$
 denote its degree in $a.$

 Enumerate  the words in the variables $a_1,\dots,a_\vmu$ 
 of degree at most $d_a$
 as $\{m_1,\dots,m_N\}.$  In particular, 
 $N=\sum_{j=0}^{d_a} \vmu^j.$ For $1\le j,k\le \vmu$ and 
 $1\le \rR,\tT \le N$  there exist uniquely determined
 $d\times d$ matrix-valued polynomials 
 $\mfZ^{j,k}_{\rR,\tT}(a)$ 
  such that 
\begin{equation}
 \label{e:qax}
q(a,x) = \sum_{j,k,\rR,\tT}  (I_d\otimes m_{\rR}(a)^*x_j) \mfZ^{j,k}_{\rR,\tT}(a)
     (I_d\otimes x_km_\tT(a)).
\end{equation}
 In fact,

\begin{equation}
 \label{e:qax-}
(I_d\otimes m_{\rR}(a)^*x_j) \mfZ^{j,k}_{\rR,\tT}(a) (I_d\otimes x_km_\tT(a)) = \sum_{s=1}^N \{q_w w : 
     w= m_{\rR}(a)^* x_j m_s(a) x_k m_{\tT}(a)\}. 
\end{equation}

 Letting $\cZ$ denote the block  matrix indexed by $((j,\rR),(k,\tT))$ with
 $d\times d$ polynomial entries $\mfZ^{j,k}_{\rR,\tT}(a)$ and 
 letting $V(a)[x]$ the column vector with $(k,\tT)$ entry 
 $I_d\otimes x_k m_\tT(a),$  equation~\eqref{e:qax} becomes,

\begin{equation}
 \label{e:qax+}
 q(a,x)  =  V(a)[x]^*  \, \cZ(a) \,  V(a)[x].
\end{equation}

The polynomial $V(a)[x]$ is the {\bf border vector}  and $\cZ(a)$ 
 is  the {\bf middle matrix} for $q.$
 Equation~\eqref{e:qax+} is the \df{border vector-middle matrix}
 representation of $q.$

  Before continuing, we pause to introduce non-hermitian
 freely noncommuting variables.  Accordingly,
 let $\chi_1,\dots,\chi_k,z_1,\dots,z_\ell,w_1,\dots,w_\ell$
 be freely noncommuting variables. Now let ${}^*$ denote
 an involution on words in these variables that reverses
 the order of products and satisfies $\chi_j^*=\chi_j$
 and $z_j^*=w_j.$ Thus the $\chi$ variables are
 hermitian, but the $z,w$ variables are not. It is natural,
 and customary, to systematically use $z_j^*$ in place
 of $w_j.$  A polynomial in this mix of  variables is now
 a linear combination of words with matrix coefficients. 
 A word in these variables evaluates at a 
 tuple $(X,Z) \in \SSnk\times M_n(\C^\ell)$  in the 
 natural way:  replace $\chi_j$ with $X_j$ 
 and similarly  replace $z_j$ and $z_j^*$ with 
 $Z_j$ and $Z_j^*.$  The involution 
 extends in the evident fashion to this mixed variable
 setting. Namely, the coefficient matrices are replaced
 by their adjoints and the involution is applied to the words. 
 Finally, a polynomial is hermitian if $p^*=p$; equivalently
 $p(X,Z)^*=p^*(X,Z)$ for all tuples $(X,Z).$
 
 The definition of the border vector, as a polynomial,
 naturally extends to the case of non-hemitian $x$ variables. 
 With this understanding, and    given positive integers $m,n,$ a tuple
 $B\in \SSnmu,$ a tuple $\beta\in M_{n,m}(\C^\mu)$ and
 tuple $\alpha\in \mathbb{S}_m(\C^\mu),$ 
\begin{equation}
 \label{e:qBba}
\begin{split}
  \sum_{j,k,\rR,\tT} &  (I_d\otimes m_{\rR}(B)^*\beta_j) \mfZ^{j,k}_{\rR,\tT}(\alpha)
     (I_d\otimes \beta_k^* m_\tT(B))\\
&= V(B)[\beta^*]^* \cZ(\alpha) V(B)[\beta^*].
\end{split}
\end{equation}

\begin{proposition}
 \label{p:BVMM}
If  $W,R,S$ are given as in  Proposition~\ref{thm:partialconvexity} item~\ref{i:a2convex3}, then
\begin{equation*}
  (I_d\otimes W)^* \, q(R,S)\,   (I_d\otimes W)
 = q(A,X) + V(A)[\beta^*]^* \, \cZ(\alpha)\,  V(A)[\beta^*].
\end{equation*}
\end{proposition}

\begin{proof}
 Suppose  $w = \ell(a)x_jc(a)x_kr(a),$ where 
 $\ell(a), c(a), r(a)$ are words. 
Compute
\begin{equation*}
\begin{split}
w(R,S) =  & \begin{pmatrix}  \ell(A) & 0 \\ 0 &  \ell(\alpha) 
\end{pmatrix} 
  \begin{pmatrix}  X_j & \beta_j \\ \beta_j^* &  \delta_j \end{pmatrix}
 \begin{pmatrix}  c(A) & 0 \\ 0 & c(\alpha)
\end{pmatrix} 
  \begin{pmatrix}  X_k & \beta_k \\ \beta_k^* &  \delta_k \end{pmatrix}
 \begin{pmatrix}  r(A)  & 0 \\ 0 & r(\alpha)\end{pmatrix}  \\
& = \begin{pmatrix} \ell(A)X_jc(A)X_kr(A) + \ell(A)\beta_jc(\alpha)\beta_k^*r(A) &  \quad * \\ \quad * & \quad *   \end{pmatrix}.
\end{split}
\end{equation*}
Hence, %
\begin{equation}
\label{eq:imp1}
 W^*\,  w(R,S)\,  W  = 
\ell(A)X_jc(A)X_kr(A)  + \ell(A)\beta_j c(\alpha) \beta_k^*r(A).
\end{equation}
 In particular, fixing $\rR,\tT,j,k$ and letting $Y=I_d\otimes W,$
 equations~\eqref{e:qax-} and \eqref{eq:imp1} give
\begin{equation}
 \label{e:ohmy}
\begin{split}
Y^*  (I_d\otimes & S_j m_{\rR}(R))^*  \,  \mfZ^{j,k}_{\rR,\tT}(R)\, 
  (I_d\otimes S_k m_{\tT}) (R))Y \\
=  Y^* &  \left (\sum_{s=1}^N  \{ q_w \otimes w(R,S):
   w=m_{\rR}(a)x_j m_s(a) x_k m_{\tT}(a)\} \right ) Y \\
 =  &  \sum_{s=1}^N  \{ q_w \otimes [w(A,X)
     + m_{\rR}(A)^*\beta_j m_s(\alpha) \beta_k^* m_{\tT}(A)]: 
   w=m_{\rR}^*x_j m_s x_k m_{\tT}\}  \\
= & (I_d\otimes X_j m_{\rR}(A))^* \, \mfZ^{j,k}_{\rR,\tT}(A)  \, 
  (I_d\otimes X_k m_{\tT}(A)) \\
   & + (I_d\otimes \beta_j^* m_{\rR}(A))^* 
    \, \mfZ^{j,k}_{\rR,\tT}(\alpha)\, 
     (I_d\otimes \beta_k^* m_{\tT}(A)).
\end{split}
\end{equation}
 Summing  equation~\eqref{e:ohmy} 
 over $\rR,\tT,j,k$ and using equations~\eqref{e:qax},
 \eqref{e:qax+}  and \eqref{e:qBba},
\begin{equation*}
 (I_d\otimes W)^* q(R,S)  (I_d\otimes W)
  = q(A,X) + V(A)[\beta^*]^* \cZ(\alpha) V(A)[\beta^*]. 
\qedhere
\end{equation*}
\end{proof}

\subsection{Proof of Theorem \ref{thm:partialconvexpolys}}
\label{s:a2endproof}
\begin{proof}[Proof of Theorem \ref{thm:partialconvexpolys}]
Since $p$ is convex in $x,$ Proposition~\ref{p:a2convex-deg2}
 says  its degree in $x$ is at
 most two. Thus,
\begin{equation*}
  p(a,x) = L(a,x) + q(a,x),
\end{equation*}
 where $L(a,x)$ is affine linear in $x$ and 
\begin{equation*}
 q(a,x) =\sum_{w\in \Gamma} p_w w,
\end{equation*}
 where $\Gamma$ denotes words in the variables $a,x$ that are homogeneous
 of degree two in $x.$
  Since $p$ is convex in $x,$ so is $q,$ 
  and it suffices
 to prove that there exists an $xy$-pencil $\Lambda$
 such that $q=\Lambda^* \Lambda.$

  Let $\kappa$ denote the degree of $q$ in $a.$ 
   By Proposition~\ref{rem:betterlinindep}, 
  there is an $\ell$ such that for all $n\ge \ell$
  there exists an $A\in \SSnmu$ and a $v\in \C^n$ such that
\[
  \cM_{A,v, \kappa} 
    =\{w(A)v: w \mbox{ is a word of length at most } \kappa\}
\]
 is linearly independent. 
 
 Fix $n\ge \ell$ and choose $C \in \SSnmu, v \in \C^n$ such that $\cM_{C,v}$
 is linearly independent. For this $C$ and a given $H\in M_n(\C^\mu),$ 
 the border vector evaluated at $(C,H^*)$ is
\[
V(C)[H^*] = \displaystyle \bigoplus_{j=1}^\mu \begin{pmatrix}  H^*_j m_1 (C) \\
 \vdots \\  H_j^* m_N(C) \end{pmatrix}. 
\]
 By linear independence of $\cM_{C,v},$ 
\begin{equation}
\label{e:spanisall}
 \{V(C)[H^*]v:H\in M_n(\C^\mu)\} = \C^{\mu n N}.
\end{equation}

 Let $\alpha \in \SSnmu$ 
 be given.  To prove that 
 $\cZ(\alpha)\in M_d(\C)\otimes M_{\mu n N}(\C)$ is positive
 semidefinite,  let $z\in \C^d\otimes \C^{\mu n N}$ be given.
 There exist  $\gamma_1,\dots,\gamma_d\in \C^d$
 and $u_1,\dots,u_d\in \C^{\mu n N}$ such that 
 $z=\sum \gamma_\jj \otimes u_\jj.$ 
 By equation \eqref{e:spanisall}, 
  for each $1\le \jj\le d,$ there exist $H^\jj  \in M_n(\C^\mu)$
 such that $u_{\jj} = V(C)[(H^\jj)^*]v.$
 Let $\beta_j$ denote the $d\times 1$ block matrix with
 $(a,1)$ entry $H_j^\jj.$  Thus $\beta_j\in M_{dn,n}(\C)$
 and $\beta \in M_{dn,n}(\C^\mu).$
 Let   $v_\jj=e_\jj\otimes v  \in \C^d\otimes \C^{n},$
 where $\{e_1,\dots,e_d\}$ is the standard orthonormal 
 basis for $\C^d.$ 

 Set $A=I_d\otimes C\in \mathbb{S}_{dn}(\C^\mu).$
 Thus, $A$ is the direct sum of $C$ with itself  $d$-times. 
 Let $\Gamma=\sum_{\kk=1}^d  \gamma_\kk\otimes v_\kk$
   and compute
\begin{equation}
 \label{e:AtoC}
\begin{split}
  (I_d\otimes V(A)[\beta^*]) & \Gamma 
 =  \sum_{\kk=1}^d \gamma_\kk \otimes V(A)[\beta^*](e_\kk\otimes v)\\
 = &   \sum_{\kk=1}^d \gamma_\kk \otimes V(C)[(H^\kk)^*]v 
 =  \sum_\kk^d  \gamma_\kk \otimes u_\kk  = z.
\end{split}
\end{equation}

 Let $\delta\in \SSnmu$ be given 
 and let $W,R,S$ have the form given in 
 Proposition~\ref{thm:partialconvexity} item~\ref{i:a2convex3}.
  Since $q$ is convex in $x,$ item~\ref{i:a2convex3}
  of Proposition~\ref{thm:partialconvexity} implies
\begin{equation}
\label{eq:imp2-alt}
(I_d \otimes W)^*\,  [q(R,S)]\,  (I_d \otimes W)  \succeq q(A,X). 
\end{equation}
 Proposition~\ref{p:BVMM} and equation~\eqref{eq:imp2-alt} give
\begin{equation}
\label{eq:imp3}
(I_d \otimes V(A)[\beta^*])^* \, \mathcal Z(\alpha) \,
  (I_d \otimes V(A)[\beta^*]) \succeq 0.
\end{equation}
 Combining equations~\eqref{eq:imp3} and \eqref{e:AtoC} gives,
\[
 0 \le \langle \cZ(\alpha) (I_d\otimes V(A)[\beta^*]) \Gamma,
   (I_d\otimes V(A)[\beta^*]) \Gamma\rangle
 = \langle \cZ(\alpha) z,z\rangle
\]
 and thus $\cZ(\alpha)\succeq 0.$  

 At this point, it has been shown that there is an $\ell$
 such that if $n\ge \ell$ and $\alpha\in \SSnmu,$ then
 $\cZ(\alpha)\succeq 0.$ Hence, by a standard direct
 sum argument,  $\cZ(\alpha)\succeq 0$
  for all $n$ and   $\alpha\in \SSnmu;$ that is
 $\cZ$ is a positive polynomial. 
 Hence $\cZ$  
 factors \cite{M} in the sense 
 that  there exists a (not necessarily square) matrix polynomial 
 $F$ such that $\cZ(a)=F(a)^*F(a).$  Consequently, 
\[
q(a,x) = V(a)[x]^* \cZ(a) V(a)[x] = \Lambda(a,x)^* \Lambda(a,x),
\]
 where $\Lambda(a,x) = F(a)V(a)[x]$ is linear in $x$
  and the proof is complete. 
\end{proof}

\subsection{Biconvexity}
\label{s:a2cors}
This section concludes by collecting consequences
  of Theorem~\ref{thm:partialconvexpolys} for later use.
Let \df{$\mathcal{L}$} denote the set of words in $a,x$ of 
 degree at most two in both $a$ and $x,$ 
 but  excluding those of the forms $a_ja_ix_kx_m$
 and $x_mx_ka_ia_j.$

\begin{corollary}
 \label{c:preinfoabtp}
    Suppose $p(a,x)$ is a hermitian $d \times d$ matrix 
  polynomial. If $p$ is convex in 
 $x$ and has degree at most two in $a,$ 
  then $p$ contains no words of the form $x_jx_\ell a_ka_m$ or
 $a_ma_kx_\ell x_j;$ that is  $p(a,x) \in M_d \otimes \spann \cL.$
\end{corollary}

\begin{proof}
 From Theorem~\ref{thm:partialconvexpolys}, 
\[
 p(a,x) = L(a,x) + \Lambda(a,x)^*\Lambda(a,x),
\]
 for matrix-valued polynomials $L$ and $\Lambda,$ where 
 $L$ is affine linear in $x$ and $\Lambda$ is linear in $x.$
  Since $p$ has degree at most
 two in $a,$ it is immediate that $L(a,x)$ has degree at most two in $a$
 and thus is a (matrix-valued) linear combination of elements of $\mathcal{L}.$
 Let $N$ denote the degree of $\Lambda$ in $a$ and, 
 arguing by contradiction, suppose $N\ge 2.$  Write
\[
 \Lambda(a,x) =\sum_{u=0}^N \Lambda_u(a,x),
\]
 where $\Lambda_u(a,x)$ is homogeneous of degree $u$ in $a.$ By assumption
 $\Lambda_N(a,x)\ne 0.$ Hence, by Lemma~\ref{cor:gsforus}, 
 there exist $A,X$ such that $\Lambda_N(A,X)\ne 0.$ It follows that
 the matrix-valued polynomial of the single real variable $t,$
\[
 F(t) =\Lambda(tA,X)  =\sum_{u=0}^N  t^u\Lambda_u(A,X)
\]
 has degree $N.$ Hence
\[
p(tA,X) = L(tA,X) + \Lambda(tA,X)^*\Lambda(tA,X) = L(tA,X)  + F(t)^*F(t)
\]
 has degree $2N\ge 4$ in $t,$ contradicting the assumption that $p$
 has degree at most two in $A.$   We conclude that $\Lambda(a,x)$
 has degree at most one in both $a$ and $x$ and the proof is complete.
\end{proof}

\begin{corollary}
\label{prop:infoabtp}
 Suppose $p(a,x)$ is a hermitian $d \times d$ matrix  polynomial. If $p$ is convex in 
 both $a$ and $x$  (separately), then $p$ has degree at most two 
 in both $a$ and $x$ and  contains no words of the form $x_jx_\ell a_ka_m$ or
$a_ma_kx_\ell x_j;$ that is  $p(a,x) \in M_d \otimes \spann \mathcal L.$
\end{corollary}

\begin{proof}
If the hermitian polynomial $p(a,x)$ is convex in both  $a$ and $x,$ then
 Theorem~\ref{thm:partialconvexpolys} holds with the roles of $a$ and $x$
interchanged.  In particular, if $p$ is convex in both $a$ and $x,$
 then $p$ has degree at most two in both $a$ and $x$ and 
this result thus follows from Corollary~\ref{c:preinfoabtp}.
\end{proof}

\section{xy-convex  hermitian polynomials}
Proposition~\ref{p:more-is-true} and  Theorem~\ref{t:introxyconvexp}
  are  proved in this section. 
 The proof strategy
 is to show $xy$-convexity 
 here is to proceed directly from Corollary~\ref{prop:infoabtp}. 
 For notational consistency with \cite{JKMMP21} we use
 $x=(x_1,\dots,x_\mu)$ and $y=(y_1,\dots,y_\mu)$ instead
 of $a,x$ for the two classes of variables.


\begin{proposition}[[Proposition 4.1, JKMMP21]
 \label{p:xy-convexp-alt}
  A triple $((X,Y),V)$ is an $xy$-pair if and only if, up to unitary equivalence,
  it has the block form
\begin{equation}
\label{e:xypairform}
 X_j=\begin{pmatrix} X_{0j} & A_j&0\\ A_j^* & * & * \\ 0 & * & * \end{pmatrix}, \ \
 Y_k= \begin{pmatrix} Y_{0k} & 0&C_k\\ 0& *&*\\ C_k^* & * & * \end{pmatrix}, \ \ 
 V= \begin{pmatrix} I& 0&0\end{pmatrix}^*,
\end{equation}
 $1 \le j,k \le \mu.$
 Thus, a polynomial $p(x,y)\in M_d(\C \langle x,y \rangle)$ is $xy$-convex 
  if and only if
\[
 (I_d \otimes V)^*p(X,Y)( I_d \otimes V) - p(X_0,Y_0) \succeq 0
\]
 for each $xy$-pair $((X,Y),V)$ of the form of 
 equation \eqref{e:xypairform}. 
\end{proposition}

Recall the definition of $\cL$ from Subsection~\ref{s:a2cors}.

%

\begin{proof}[Proof of Proposition~\ref{p:more-is-true}] 
To show that $p$ is convex in $x$ and $y$ separately,  simply  replace
$X_1,X_2,Y$ in the proof \cite[Lemma~4.3]{JKMMP21} with 
  $X^1,X^2,Y\in \mathbb S_n(\C^\mu).$   

 To prove item \ref{i:more2} implies item \ref{i:more3}, 
let $\cW_1$ denote the words of degree at most one in each of 
 $x$ and $y$ 
 separately, and let $\cW_2$ denote the set of words 
 that have degree at least two, but no more than two in each
 of $x,y,$ but contains none of the words of the form
$x_jx_\ell y_ky_m$ or $(x_jx_\ell y_ky_m)^*,$
 for  $1 \le j,k,\ell, m \le \mu$. 
 
Since $p(x,y)$ is convex in $x$ and $y$ separately, 
from Corollary~\ref{prop:infoabtp},   $p$ has the form,
\[
p(x,y) = l(x,y) + q(x,y),
\]
where 
\[
 \ell(x,y) =\sum_{w\in \cW_1} p_w w, \ \ \
 q(x,y)=\sum_{w\in \cW_2} p_w w,
\]
 for some $p_w\in M_d(\C).$ 

Let $\cW_{2,x}$ denote those words in $\cW_2$ that have degree two in $x.$ 
Define $\cW_{2,y}$ similarly. A computation shows
\[
 \frac12 p_{x,x}(x,y)[x]=\frac12 q_{x,x}(x,y)[x] =\sum_{w\in \cW_{2,x}}   p_w w;
\]
 that is,
\begin{align*}
\frac 1 2 p_{xx}(x,y)[x] & = \frac 1 2 q_{xx}(x,y)[x] =  \sum_{j,k, \ell, m=1}^\mu [  p_{x_jx_\ell} x_j x_\ell\\
& + p_{x_j x_\ell y_k}x_jx_\ell y_k + p_{y_k x_\ell x_j} y_k x_\ell x_j  +   p_{x_jy_kx_\ell} x_j y_k x_\ell \\
& +  p_{x_jy_ky_mx_\ell} x_jy_ky_mx_\ell + p_{y_k x_j x_\ell y_m} y_k x_j x_\ell y_m \\
& + p_{x_jy_kx_\ell y_m} x_jy_kx_\ell y_m + p_{y_kx_jy_m x_\ell} y_kx_jy_m x_\ell]
\end{align*}
 Similarly, 
\[
 \frac12 p_{y,y}(x,y)[y]=\frac12 q_{y,y}(x,y)[y] =\sum_ {w\in \cW_{2,y} } p_w w
\]
Since $p(x,y)$ is convex in $x$ and $y$ separately, 
the partial Hessian of $p$ with respect to $x$ as well as $y$ is positive. 
In particular,
\begin{equation}
 \label{e:partials-positive}
p_{xx}(x,y)[x], \ \ p_{yy}(x,y)[y] \succeq 0.
\end{equation}

Let $\cW_{1,x}$ denote those words in $\cW_1$ that have
 degree one in $x.$ Define $\cW_{1,y}$ similarly. 
By  \cite[Theorem~0.2]{M}, the 
 positivity condition in equation~\eqref{e:partials-positive}
 implies there exist an 
$N$ and $N \times d$ matrix-valued free polynomials $f(x,y)$  and $g(x,y)$ such that 
\[
p_{xx}(x,y)[x] = f(x,y)^*f(x,y) \,\, \, \,  p_{yy}(x,y)[y] = g(x,y)^*g(x,y),
\] 
where 
\[
f(x,y) =  \sum_{w\in\cW_{1,x}} f_w w = \sum_{j,k=1}^\mu  f_{x_j} x_j + f_{x_jy_k}x_jy_k + f_{y_kx_j}y_kx_j 
\]
 and similarly, $g(x,y) =\sum_{w\in \cW_{1,y}} g_w w.$

Let $\cW_{1,x,y}$ denote the words of degree one in both $x$ and $y.$
 Let $x$ and $y$ denote the column vectors 
\[
 x=\begin{pmatrix} x_j \end{pmatrix}_{j=1}^\mu, \ \ \
 y=\begin{pmatrix} y_j\end{pmatrix}_{j=1}^\mu
\]
 and let {$v$} denote the column vector
\[
 v= \begin{pmatrix} w \end{pmatrix}_{w\in \cW_{1,x,y}}.
\]
 Let
\[
 \cW_x=\begin{pmatrix} x \\ v \end{pmatrix}, \ \ \
 \cW_y =\begin{pmatrix} v \\ y \end{pmatrix}.
\]
 Likewise, let $F_0,F_1$ and $F$ denote the row vectors,
\[
 F_0 =\begin{pmatrix} f_{x_j}\end{pmatrix}_{j=1}^\mu, \ \ \
 F_1 =\begin{pmatrix} f_{w} \end{pmatrix}_{w\in \cW_{1,x,y}}, \ \ \
 F=\begin{pmatrix} F_0&F_1 \end{pmatrix}
\]
 and similarly
\[
 G_0 = \begin{pmatrix} g_{x_j}\end{pmatrix}_{j=1}^\mu, \ \ \
 G_1 =\begin{pmatrix} g_{w} \end{pmatrix}_{w\in \cW_{1,x,y}}, \ \ \
 G=\begin{pmatrix} G_1&G_0 \end{pmatrix}.
\] 
 Thus,
\[
  f = F\cW_x, \ \ \ g=G\cW_y
\]
 and 
\[
 \cW_x^*F^*F \cW_x = f^*f, \ \
 \cW_y^*G^*G \cW_y = g^*g.
\]
 Let
\[
 P =\begin{pmatrix}{p_{u^*w} }\end{pmatrix}_{w\in \cW_{1,x,y}}
\]
 and observe that
\[
\begin{split}
 F^*F &
   =\begin{pmatrix} F_0^*F_0 &  F_0^*F_1 \\ F_1^*F_0 & P \end{pmatrix}, \, 
G^*G =\begin{pmatrix} P & G_1^*G_0 \\ G_0^*G_1 & G_0^*G_0\end{pmatrix}.
\end{split}
\]

 Let
\[
 \cM =\begin{pmatrix} F_0^*F_0 & F_0^*F_1 & 0 \\ F_1^*F_0 & P & G_1^*G_0\\
  0& G_0^*G_1 & G_0^*G_0 \end{pmatrix}, \ \
 \cW=\begin{pmatrix} x \\ v\\ y\end{pmatrix},
\]
 and observe 
\[
 q(x,y)=  \cW^* \, \cM \, \cW.
\]
 Since $F^*F$ and $G^*G$ are positive semidefinite,
 \cite[Proposition~1]{T} implies there is a 
$d \mu \times d \mu$ matrix $Q$ 
such that 
\[
\widehat{\mathcal M} =  
 \cM +\begin{pmatrix}0&0&Q\\0&0&0\\Q^* &0&0\end{pmatrix} =
 \begin{pmatrix} F_0^*F_0 & F_0^*F_1 & Q \\ F_1^*F_0 & P & G_1^*G_0\\
  Q^* & G_0^*G_1 & G_0^*G_0 \end{pmatrix}
\succeq 0.
\]

Letting $Q = \begin{pmatrix}  Q_{j,k}\end{pmatrix}_{j.k=1}^\mu \in M_\mu \otimes M_d$, it 
follows that 
\[
p(x,y) = \lambda(x,y) + 
 \mathcal W^* \widehat{\mathcal M} \mathcal W,
\]
where 
\begin{align*}
\lambda(x,y) &= l(x,y) - \left \{\sum_{j,k=1}^\mu Q_{j,k} x_jy_k + Q_{j,k}^*y_kx_j \right\} \\
          & = \sum_{j,k=1}^ \mu \left[p_{x_j} x_j +  p_{y_j}y_j + 
 (p_{x_jy_k} - Q_{j,k}) x_jy_k + (p_{y_kx_j} - Q_{j,k}^*) y_kx_j\right].
\end{align*}
Since $\widehat{\mathcal M} \succeq 0$, there exists a matrix $\mathcal R$ such that 
$\widehat{\mathcal M} = \mathcal R^* \mathcal R$. Finally, letting $\Lambda(x,y) = \mathcal R \mathcal W$, it follows that $\Lambda(x,y)$ is a $d \times d$ matrix-valued $xy$-pencil and 
\[
p(x,y) = \lambda(x,y)  + \Lambda(x,y)^*\Lambda(x,y). 
\]
To prove item~\ref{i:more3} implies item~\ref{i:more1},
let a triple $((X,Y),V)$ as in Proposition \ref{p:xy-convexp-alt} be given and observe,
\begin{align*}
 p(V^*(X,Y)V) & = \lambda(V^*(X,Y)V) + \Lambda(V^*(X,Y)V)^* \Lambda (V^*(X,Y)V) \\
                         & = (I_d \otimes V)^*\lambda(X,Y)(I_d \otimes V) + (I_d \otimes V)^* \Lambda(X,Y)^* (I_d \otimes V V^*) \Lambda(X,Y) (I_d \otimes V) \\
                        & \preceq (I_d \otimes V)^*\lambda(X,Y)(I_d \otimes V) + (I_d \otimes V)^* \Lambda(X,Y)^*   \Lambda(X,Y) (I_d \otimes V) \\
                        & = (I_d \otimes V)^*p(X,Y) (I_d \otimes V).
\end{align*}
It follows from Proposition \ref{p:xy-convexp-alt} that $p$ is $xy$-convex and the proof is complete. \qedhere 
\end{proof}

It is well known that sums of squares representations can be certified with a semidefinite program (and hence are tractable).   See for instance \cite{KP,KMP,BS,BKP}.
  For simplicity, {we treat the case $\mu =  1.$} 
 Thus, by Proposition~\ref{prop:infoabtp}, 
\[
 p(x,y) = \sum p_w w,
\]
 where the sum is over  words in $x,y$ of 
 degree at most two in both $x$ and $y,$ 
 but  excluding $x^2y^2$ and $y^2x^2.$ 
 Let



\begin{proposition}
 \label{prop:SDP}
   The polynomial $p$ is $xy$-convex if and only if there is a positive semidefinite 
  $Q=(Q_{j,k})_{j,k=1}^4$  such that 
\begin{enumerate}
 \item $Q_{j,k}=p_{w_jw_k}$  for $1 \le j,k \le 2$ and $3\le j,k\le 4;$
 \item $Q_{1,3}+ Q_{3,1}= p_{w_1 w_3^*};$ and 
 \item $Q_{2,4}+ Q_{4,2}=p_{w_2 w_4^*}$
\end{enumerate}
where $w_1 = x, w_2 =y, w_3 = xy$ and $w_4 = yx$.
\end{proposition}

\begin{proof}[Sketch of proof]
  Assuming such a $Q$ exists, factor $Q$ as $F^*F$ and write, 
\[
 F=\begin{pmatrix} F_1 &F_2 & F_3 & F_4 \end{pmatrix}
\]
{for some $N \times d$ matrices $F_j$.}
Let 
\begin{equation}
\label{e:Lambdaxy}
 \Lambda(x,y) =F_1 x +F_2 y + F_3 yx + F_4 xy.
\end{equation}
 Thus $\Lambda(x,y)$ is an $xy$-pencil. 
 A straightforward (but tedious) computation as in \cite{JKMMP21} verifies,
\[
{ p(x,y) - \Lambda(x,y)^*\Lambda(x,y)  = \lambda(x,y),}
\]
 where $\lambda(x,y)$ has degree at most one in each of  $x,y.$

 Now suppose $p$ has the form 
\[
 p(x,y) =\lambda(x,y) +\Lambda(x,y)^*\Lambda(x,y),
\]
 where $\lambda(x,y)$ has degree at most one in each of $x,y$ and 
{$\Lambda(x,y) = \Lambda_ x x+ \Lambda_y y +\Lambda_{yx} yx +\Lambda_{xy} xy.$}
 In this case, let
\[
 F= \begin{pmatrix} \Lambda_x & \Lambda_y & \Lambda_{yx} &  \Lambda_{xy} 
 \end{pmatrix}                    
\]
 and check that $Q=F^*F\succeq 0$ has the desired properties.
\end{proof}


\newpage

\printindex

\end{document}